\documentclass[12pt,leqno]{amsart}
\usepackage{amssymb,amsthm,amsmath,latexsym}
\usepackage[all]{xy}

\newtheorem{theorem}{\sc Theorem}[section]
\newtheorem{thm}[theorem]{\sc Theorem}
\newtheorem{lem}[theorem]{\sc Lemma}
\newtheorem{prop}[theorem]{\sc Proposition}
\newtheorem{cor}[theorem]{\sc Corollary}

\newtheorem{rem}[theorem]{\sc Remark}

\newtheorem*{thmA}{Theorem A}
\newtheorem*{thmB}{Theorem B}
\newtheorem*{thmC}{Corollary C}

\title[Non-abelian tensor product]{The order of the non-abelian tensor product of groups}
\author[Bastos]{R. Bastos}
\address{ Departamento de Matem\'atica, Universidade de Bras\'ilia,
Brasilia-DF, 70910-900 Brazil }
\email{(Bastos) bastos@mat.unb.br; (Rocco) norai@unb.br}
\author[Nakaoka]{I.\,N. Nakaoka}
\address{ Departamento de Matem\'atica, Universidade Estadual de Maring\'a, Maring\'a-PR, 87020-900 Brazil }
\email{(Nakaoka) innakaoka@uem.br}
\author[Rocco]{N.\,R. Rocco }
\subjclass[2010]{20E34, 20J06}
\keywords{Structure theorems; Finiteness conditions; Non-abelian tensor product of groups}

\begin{document}

\maketitle

\begin{abstract}
Let $G$ and $H$ be groups that act compatibly on each other. We denote by $[G,H]$ the derivative subgroup of $G$ under $H$. We prove that if the set $\{g^{-1}g^h \mid g \in G, h \in H\}$ has $m$ elements, then the derivative $[G,H]$ is finite with $m$-bounded order. Moreover, we show that if the set of all tensors $T_{\otimes}(G,H) = \{g\otimes h \mid g \in G, h\in H\}$ has $m$ elements, then the non-abelian tensor product $G \otimes H$ is finite with $m$-bounded order. We also examine some finiteness conditions for the non-abelian tensor square of groups. 
\end{abstract}

\maketitle

\section{Introduction}

Let $G$ and $H$ be groups each of which acts upon the other (on the right),
\[
G\times H \rightarrow G, \; (g,h) \mapsto g^h; \; \; H\times G \rightarrow
H, \; (h,g) \mapsto h^g
\]
and on itself by conjugation, in such a way that for all $g,g_1 \in G$ and
$h,h_1 \in H$,
\begin{equation}   \label{eq:0}
g^{\left( h^{g_1} \right) } = \left( \left( g^{g^{-1}_1}  \right) ^h \right) ^{g_1} \; \; \mbox{and} \; \; h^{\left( g^{h_1}\right) } =
\left( \left( h^{h_1^{-1}} \right) ^g \right) ^{h_1}.
\end{equation}
In this situation we say that $G$ and $H$ act {\em compatibly} on each other. The derivative of $G$ under (the action of) $H$, $[G,H]$, is defined to be the subgroup $[G,H] = \langle g^{-1}g^h \mid \ g \in G, h \in H\rangle$ of $G$. Similarly, the subgroup  $[H,G] = \langle h^{-1}h^g \mid \ h \in H, g \in G \rangle$ of $H$ is called derivative of $H$ under $G$. In particular, if $G=H$ and all actions are conjugations, then the derivative $[G,H]$ becomes the derived subgroup $G'$ of $G$. 

Schur \cite[10.1.4]{Rob} showed that if $G$ is central-by-finite, then the derived subgroup $G'$ is finite and thus, the group $G$ is a BFC-group. Neumann \cite[14.5.11]{Rob} improved Schur's theorem in a certain way, showing that the group $G$ is a BFC-group if and only if the derived subgroup $G'$ is finite, and this occurs if and only if $G$ contains only finitely many commutators. Latter, Wiegold proved a quantitative version of Neumann's result: if $G$ contains exactly $m$ commutators, then the order of the derived subgroup $G'$ is finite with $m$-bounded order \cite[Theorem 4.7]{W}. Now, the next result can be viewed as a version of Wiegold's result in the context of actions and derivatives subgroups $[G,H]$ and $[H,G]$, where $G$ and $H$ are groups acting compatibly on each other.

\begin{thmA}
Let $G$ and $H$ be groups that act compatibly on each other. Suppose that the set $\{g^{-1}g^h \mid g \in G, \ h \in H\} \subseteq [G,H]$ has exactly  $m$ elements. Then $[G,H]$ is finite, with $m$-bounded order.   
\end{thmA}

It should be noted that the structure of derivative subgroups provides important information about the structure of the non-abelian tensor product of groups (see for instance \cite{BNR,Nak,NR01,Vis,T}). In this direction, we want to describe  quantitative results for the non-abelian tensor product of groups (cf. \cite{BNR}).

Let  $H^{\varphi}$ be
an extra copy of $H$, isomorphic via $\varphi : H \rightarrow
H^{\varphi}, \; h \mapsto h^{\varphi}$, for all $h\in H$.
Consider the group $\eta(G,H)$ defined in  \cite{Nak} as
$$\begin{array}{ll} {\eta}(G,H) =  \langle
G,H^{\varphi}\ |  &
[g,{h}^{\varphi}]^{g_1}=[{g}^{g_1},({h}^{g_1})^{\varphi}], \;
[g,{h}^{\varphi}]^{h^{\varphi}_1} = [{g}^{h_1},
({h}^{h_1})^{\varphi}] , \\ & \ \forall g,g_1 \in G, \; h, h_1 \in H
\rangle . \end{array}$$
We observe that when $G=H$ and all actions are conjugations, $\eta (G,H)$ becomes the group $\nu (G)$ introduced in \cite{NR1}: 
$$\begin{array}{ll} {\nu}(G) =  \langle
G \cup G^{\varphi}\ |  &
[g_1,{g_2}^{\varphi}]^{g_3}=[{g_1}^{g_3},({g_2}^{g_3})^{\varphi}] = 
[g_1,{g_2}^{\varphi}]^{g^{\varphi}_3}, \ g_i \in G
\rangle . \end{array}$$

It is a well known fact (see \cite[Proposition 2.2]{Nak}) that the subgroup
$[G, H^{\varphi}]$ of $\eta(G,H)$ is canonically isomorphic with the {\em non-abelian
tensor product} $G \otimes H$, as defined by Brown and Loday in their seminal paper \cite{BL}, the isomorphism being induced by $g \otimes h \mapsto
[g, h^{\varphi}]$ (see also Ellis and Leonard \cite{EL}). It is clear that the subgroup $[G,H^{\varphi}]$ is normal in $\eta(G,H)$ and one has the decomposition
\begin{equation} \label{eq:decomposition}
 \eta(G,H) = \left ( [G, H^{\varphi}] \cdot G \right ) \cdot H^{\varphi},
\end{equation}
where the dots mean (internal) semidirect products. For a deeper discussion of non-abelian tensor product and related constructions we refer the reader to \cite{K,NR}.

An element $\alpha \in \eta(G,H)$ is called a {\em tensor} if $\alpha = [a,b^{\varphi}]$ for suitable $a\in G$ and $b\in H$.  We write $T_{\otimes}(G, H)$ to denote the set of all tensors (in $\eta(G,H)$). When $G = H$ and all actions are by conjugation, we simply write $T_{\otimes}(G)$ instead of $T_{\otimes}(G,G)$. The influence of the set of tensors in the general structure of the non-abelian tensor product and related constructions was considered for instance in \cite{BNR,BR1,BR2,LT,NR2}. In \cite{BNR} the authors proved that if the set of all tensors $T_{\otimes}(G,H)$ is finite, then the non-abelian tensor product $[G,H^{\varphi}]$ is finite. Here we obtain the following quantitative version: 

\begin{thmB} 
Let $G$ and $H$ be groups that act compatibly on each other. Suppose that there exist exactly $m$ tensors in $\eta(G,H)$. Then the non-abelian tensor product $[G,H^{\varphi}]$ is finite with $m$-bounded order. 
\end{thmB}

An immediate consequence of the above theorem is a quantitative version of the a well known result due to Ellis \cite{Ellis} concerning the finiteness of the non-abelian tensor product of finite groups (cf. \cite{BNR,LT,T}). See also Theorem \ref{cor.bound} and Remark \ref{rem.finite}, below. 

It is well known that the finiteness of the non-abelian tensor square $G \otimes G$, does not imply that $G$ is a finite group (and so, the group $\nu(G)$ cannot be finite). A useful result, due to Parvizi and Niroomand \cite[Theorem 3.1]{NP}, provides a sufficient condition: if $G$ is a finitely generated group in which the non-abelian tensor square is finite, then $G$ is finite (see also \cite[Remark 5]{NR2} for more details). The following result is a quantitative version of the above result and is a refinement of Theorem B in the context of the non-abelian tensor square of groups.

\begin{thmC}
Let $G$ be a group. Suppose that there exist exactly $m$ tensors in $\nu(G)$. Then, 
\begin{itemize}
\item[(a)] The non-abelian tensor square $[G,G^{\varphi}]$ is finite with $m$-bounded order. More specifically, $|[G,G^{\varphi}]| \leqslant m^{m \cdot n}$, where $n$ is the order of the derived subgroup $G'$; 
\item[(b)] Additionally, if the abelianization $G^{ab}$ is finitely generated, then the group $G$ is finite, with $m$-bounded order.   
\end{itemize}
\end{thmC}

Note that the assumption of the abelianization $G^{ab}$ to be finitely generated is necessary. For instance, the Pr\"ufer group $C_{p^{\infty}}$ is an infinite group such that $T_{\otimes}(C_{p^{\infty}}) = \{1\} = [C_{p^{\infty}},C_{p^{\infty}}^{\varphi}]$. We also obtain a list of equivalent conditions related to the finiteness of the non-abelian tensor square and the structure of the group $\nu(G)$ (see Theorem \ref{thm.finiteness}, below). 

\section{Proofs}

The following result is a consequence of
\cite[Proposition 2.3]{BL}.

\begin{prop}  \label{ident}  Let $G$ and $H$  be groups acting compatibly on each other.  The following statements hold in $\eta(G,H)$:
\begin{itemize}
\item[(a)] There exists an action of the free product $G\ast H$ on $[G,H^{\varphi}]$ so that for all $g\in G$, $h\in H$, $p\in G\ast H$:
$$[g,h^{\varphi}]^p=[g^p, (h^p)^{\varphi}];$$
\item[(b)] There are epimorphisms of groups $$\lambda:[G,H^{\varphi}] \to  [G,H], \; \mu: [G,H^{\varphi}] \to [H,G]$$ such that $([g,h^{\varphi}])\lambda = g^{-1} g^h, \  ([g,h^{\varphi}])\mu=h^{-g}h$, for each $g\in G$, $h\in H$;
\item[(c)] The actions of $G$ on $\ker (\mu)$ and of $H$ on $\ker (\lambda)$ are trivial.
\end{itemize}
\end{prop}

The next lemma is an immediate consequence from the definition of $\eta(G,H)$ and Proposition \ref{ident}(c).

\begin{lem} \label{lem.quo}
If $G$ and $H$ are groups that act compatibly on each other, then 
$\ker(\mu) \cap \ker (\lambda)$ is a central subgroup of $\eta(G,H)$;  
\end{lem}

For the reader's convenience we restate Theorem A. 

\begin{thmA}
Let $G$ and $H$ be groups that act compatibly on each other. Suppose that the set $\{g^{-1}g^h \mid g \in G, \ h \in H\} \subseteq [G,H]$ has exactly  $m$ elements. Then the derivative subgroup $[G,H]$ is finite with $m$-bounded order.   
\end{thmA}

\begin{proof}
Put $D=\{g^{-1}g^h \mid g \in G, \ h \in H\}$. For $g\in G$ and $h,k\in H$, let us write $[g,h]=g^{-1}g^h$ and $[g,h,k]=[[g,h],k]$. The compatibility of the actions gives us that $[g,h]^x=[g^x,h^x]$, for all $x,g \in G$ and $h \in H$. Thus, $D$ is a normal subset of $[G,H]$ and, as $|D|=m$, for each $\delta \in D$ we have $[[G,H]:C_{[G,H]}(\delta)]\leq m$. Consequently,  $\bigcap_{\delta \in D}C_{[G,H]}(\delta)$ has finite $m$-bounded index in $[G,H]$ and, by \cite[Theorem 4.7]{W}, the derived subgroup $[G,H]'$ is finite with  $m$-bounded order. Without loss of generality we may assume that $[G,H]$ is abelian. 
Since for all $x,g \in G$, $h,k \in H$, we have $ [[g,h],k]^x=[[g^x,h^x],k^x]$ 
and 
\[ [[g,h],k]^2=([g,h]^{-1}[g,h]^k)^2 =
[g,h]^{-2}[g,h]^{2k}= [[g,h]^2, k] \in D,
\]
we conclude that the abelian finitely generated subgroup $[[G,H],H]$ is normal in $G$ and each generator of this subgroup has $m$-bounded order. From this we deduce that $[[G,H],H]$ is finite with $m$-bounded order and we may assume that $H$ acts trivially on $[G,H]$. Hence, for all $g \in G$ and $h\in H$, 
\[ [g,h]^2=[g,h][g,h]^h=g^{-1}g^hg^{-h}g^{h^2}=[g,h^2] \in D.
\]
Since $|D|=m$, it follows that every element $[g,h]$ has finite $m$-bounded order. We conclude that the order of the derivative subgroup $[G,H]$ is $m$-bounded. The proof is complete.
\end{proof}

\begin{rem} \label{rem.derivatives}
Since $[G,H]$ and  $[H,G]$ are epimorphic images of the non-abelian tensor product $[G,H^{\varphi}]$, the finiteness of  $[G,H^{\varphi}]$ implies that $[G,H]$ and  $[H,G]$ are finite. However,  the converse does not hold in general. In fact, let $F_m$ and $F_n$ be free groups of finite ranks $m$ and $n$, respectively, where $m,n \geq 1$ and suppose that these groups act trivially on each other. Thus $[G,H]=\{1\}$ and  $[H,G]=\{1\}$ are finite, but by \cite[Proposition 2.4]{BL}, $[F_m, (F_n)^{\varphi}]\cong  (F_m)^{ab} \otimes_{\mathbb Z} (F_n)^{ab}$, which  is not finite. 
\end{rem}

Now we will deal with Theorem B: {\it Let $G$ and $H$ be groups that act compatibly on each other. Suppose that there exist exactly $m$ tensors in $\eta(G,H)$. Then the non-abelian tensor product $[G,H^{\varphi}]$ is finite with $m$-bounded order.}

\begin{cor} \label{cor.m-bound}
Let $G$ and $H$ be groups that act compatibly on each other. Suppose that the sets $\{g^{-1}g^h \mid g \in G, \ h \in H\} \subseteq [G,H]$ and $\{h^{-1}h^g \mid g \in G, \ h \in H\} \subseteq [H,G]$  have at most $m$ elements. Then the index $n = |[G,H^{\varphi}]: \ker(\lambda) \cap \ker(\mu)|$ is finite and $m$-bounded. 
\end{cor}
\begin{proof}
By Theorem A, both derivative subgroups $[G,H]$ and $[H,G]$ are finite groups with $m$-bounded orders. Since $|[G,H^{\varphi}]: \ker(\lambda)|=|[G,H]|$ and $|[G,H^{\varphi}]: \ker(\mu)|=|[H,G]|$, it follows that $\ker(\lambda) \cap \ker(\mu)$ has index at most $|[G,H]| \cdot |[H,G]|$. The proof is complete.  
\end{proof}

\begin{lem} \label{lem.finite}
Let $G$ and $H$ be groups that act compatibly on each other. Suppose that there are exactly $m$ tensors in $\eta(G,H)$. Then for every $x \in G$ and $y \in H$ we can write: $$[x,y^{\varphi}]^{n+1} = [x,(y^2)^{\varphi}][x^ y,y^{\varphi}]^{n-1},$$ where $n = |[G,H^{\varphi}]/(\ker(\mu)\cap \ker(\lambda))|$.
\end{lem}

\begin{proof} Since $|T_{\otimes}(G,H)|=m$, each of the sets $\{g^{-1}g^h \mid g \in G, h \in H\}$ and $\{h^{-1}h^g\mid g \in G, h \in H\}$ has at most $m$ elements. By Theorem A, the derivative subgroups $[G,H]$ and  $[H,G]$ are finite with $m$-bounded order. Moreover, the index $|[G,H^{\varphi}]: \ker(\mu) \cap \ker(\lambda)| = n$ is finite (Corollary \ref{cor.m-bound}). We conclude that for every $x,y \in G$ the element $[x,y^{\varphi}]^n \in \ker(\mu) \cap \ker(\lambda)$. Thus, by Lemma \ref{lem.quo},  $[x,y^{\varphi}]^n \in  Z(\eta(G,H))$ and so, $[x,y^{\varphi}]^{n+1} = x^{-1}(y^{-1})^{\varphi}x[x,y^{\varphi}]^{n}y^{\varphi}$. Further, 
\begin{eqnarray*}
[x,y^{\varphi}]^{n+1} & = & x^{-1}(y^{-1})^{\varphi}x[x,y^{\varphi}]^{n}y^{\varphi} \\
 & = & [x,(y^2)^{\varphi}] (y^{-1})^{\varphi} [x,y^{\varphi}]^{n-1}y^{\varphi} \\
 & = & [x,(y^2)^{\varphi}] ([x,y^{\varphi}]^{n-1})^{y^{\varphi}}\\
 & = & [x,(y^2)^{\varphi}] [x^y,y^{\varphi}]^{n-1}, \ \text{by definition of $\eta(G,H)$},
\end{eqnarray*}
which establishes the formula.
\end{proof}

We are now in a position to prove Theorem B.

\begin{proof}[Proof of Theorem B]
By Lemma \ref{lem.quo}, the subgroup $\ker(\mu) \cap \ker(\lambda)$ is a central subgroup of $\eta(G,H)$. Set $N=\ker(\mu) \cap \ker(\lambda)$ and $n = |[G,H^{\varphi}]/N|$. By Corollary \ref{cor.m-bound}, the index $n$  is $m$-bounded. We claim that every element in $[G,H^{\varphi}]$ can be written as a product of at most $\displaystyle{m \cdot n}$ tensors. Indeed, suppose that an element $\alpha \in [G,H^{\varphi}]$ can be expressed as a product of $r$ tensors but cannot be written as a product of fewer tensors. If $r> m \cdot n$, then one of the tensors must appear in the product at least $n+1$ times. In particular, since the set of tensors is normal and by definition of $\eta(G,H)$, $[g,h^{\varphi}]^{x} = [g^x,(h^x)^{\varphi}]$ and  $[g,h^{\varphi}]^{y^{\varphi}} = [g^y,(h^y)^{\varphi}]$, for all $g,x\in G$ and $h,y \in H$, we can write $$ \alpha = [a,b^{\varphi}]^{n+1}[a_{n+2},b_{n+2}^{\varphi}]\ldots [a_{r},b_{r}^{\varphi}],$$ where $a,a_{n+2},\ldots,a_r \in G$ and $b,b_{n+2},\ldots,b_r \in H$. By Lemma \ref{lem.finite}, $$[a,b^{\varphi}]^{n+1} = [a,(b^2)^{\varphi}][a^b,b^{\varphi}]^{n-1}.$$ It follows that $\alpha$ can be rewritten as a product of $r-1$ simple tensors, contrary to the minimality of $r$. From this we conclude that $r \leqslant m \cdot n$. Now, since there exists at most $m$ simple tensors, we conclude that $\vert [G,H^{\varphi}] \vert \leqslant m^{m\cdot n}$, as well. In particular, $[G,H^{\varphi}]$ is finite with $m$-bounded order. The proof is complete. 
\end{proof}

In \cite{M}, Moravec proved that if $G$ and $H$ are locally finite groups of finite exponent  acting compatibly on each other, then there is a bound to the exponent of the non-abelian tensor product $G \otimes H$ in terms of the exponent of the involved groups. This bound depends to the positive solution of the restricted Burnside problem (Zel'manov, \cite{ze1,ze2}). Using the general description of the group $\eta(G,H)$ we present an explicit bound to the exponent of the non-abelian tensor product of groups, when $G$ and $H$ are finite groups. Moreover, we present another proof of Ellis' result \cite{Ellis}.  

\begin{thm} \label{cor.bound}
Let $G$ and $H$ be finite groups that act compatibly on each other. Then the non-abelian tensor product $[G,H^{\varphi}]$ is finite. Moreover, the exponent $\exp([G,H^{\varphi}])$ is finite and $\{|G|,|H|\}$-bounded.   
\end{thm}

\begin{proof}
By Lemma \ref{lem.quo}, $\ker(\mu) \cap \ker(\lambda)$ is a central subgroup of $\eta(G,H)$. Set $n = [[G,H^{\varphi}]:\ker(\mu) \cap \ker(\lambda)]$. Note that $n$ divides $|G|\cdot |H|$, because $[G,H] \leqslant G$ and $[H,G] \leqslant H$. Since $|\eta(G,H)/(\ker(\mu) \cap \ker(\lambda))| = |G|\cdot |H|\cdot n$, it follows that the derived subgroup $\eta(G,H)'$ is finite and $\exp(\eta(G,H)')$ divides $|G| \cdot |H| \cdot n$ (Schur's theorem \cite[10.1.4]{Rob}). In particular, the non-abelian tensor product $[G,H^{\varphi}]$ is finite and $\exp([G,H^{\varphi}])$ divides $|G| \cdot |H| \cdot n$. The proof is complete.   
\end{proof}

\begin{rem} \label{rem.finite}
Since the proof of the above result is based on the general structure of $\eta(G,H)$ (cf. \cite{Nak}) and on Schur's theorem \cite[10.1.4]{Rob}, it becomes evident that it provides only a crude bound to both, the order and the exponent of the non-abelian tensor product $[G,H^{\varphi}]$. However, the advantages of these results are the explicit limits and the elementary proofs (without using homological methods). See \cite{M} for more details. Recently, other proofs of this result which are of non-homological nature have appeared (see for instance \cite{BNR,LT,T}).    
\end{rem}

The remainder of this section will be devoted to obtain finiteness conditions for the non-abelian tensor square of groups. 

\begin{lem} \label{lem.abelianization}  \cite[Theorem C, (a)]{BNR}
Let $G$ be a group with finitely generated abelianization. Assume that the diagonal subgroup   $\Delta(G)$ is periodic. Then the abelianization $G^{ab}$ is finite. Moreover, $G^{ab}$ is isomorphic to some subgroup of  $\Delta(G)$. 
\end{lem}

For the reader's convenience we restate Corollary C: \\ 

\noindent {\bf Corollary C.}{ Let $G$ be a group. Suppose that there exist exactly $m$ tensors in $\nu(G)$. Then, 
\begin{itemize}
\item[(a)] The non-abelian tensor square $[G,G^{\varphi}]$ is finite, with $m$-bounded order. More specifically, $|[G,G^{\varphi}]| \leqslant m^{m \cdot n}$, where $n$ is the order of the derived subgroup $G'$;
\item[(b)] Additionally, if the abelianization $G^{ab}$ is finitely generated,  then the group $G$ is finite, with $m$-bounded order.   
\end{itemize}
}
\begin{proof}
\noindent (a). Applying Theorem B to $[G,G^\varphi]$ we deduce that the order of the non-abelian tensor square is finite with $m$-bounded order. Arguing as in the proof of Theorem B we conclude that $|[G,G^{\varphi}]|\leq m^{m\cdot n}$. \\ 

\noindent (b). By the previous item, the non-abelian tensor square  $[G,G^{\varphi}]$ and the derived subgroup $G'$ are finite with $m$-bounded orders. Now, it suffices to prove that the abelianization is finite with $m$-bounded order. By Lemma \ref{lem.abelianization}, the abelianization $G^{ab}$ is isomorphic to a subgroup of the diagonal subgroup $\Delta(G)$. Since $\Delta(G) \leqslant [G,G^{\varphi}]$, it follows that $\Delta(G)$ is finite with $m$-bounded order. The proof is complete. 
\end{proof}

It should be noted that the next result makes evident an interesting relation between the constructions $\nu(G)$ and the non-abelian tensor square $G \otimes G$. More precisely, we collect a list of equivalences which give a relation between the set of commutators of the group $\nu(G)$ and the set of tensors $T_{\otimes}(G)$. 

\begin{thm} \label{thm.finiteness}
Let $G$ be a group. The following properties are equivalents.
\begin{itemize}
\item[(a)] $\nu(G)$ is a BFC-group; 
\item[(b)] The set of all commutators $\{[\alpha,\beta] \mid \alpha,\beta \in \nu(G)\}$ is finite; 
\item[(c)] The derived subgroup $\nu(G)'$ is finite;
\item[(d)] The non-abelian tensor square $[G,G^{\varphi}]$ is finite;
\item[(e)] $G$ is a BFC-group and $G^{ab} \otimes_{\mathbb{Z}}G^{ab}$ is finite;
\item[(f)] The set of tensors $T_{\otimes}(G) = \{[g,h^{\varphi}] \mid g,h \in G\} \subseteq \nu(G)$ is finite. 
\end{itemize}
\end{thm}
\begin{proof}
The equivalences $(a) \Leftrightarrow (b) \Leftrightarrow (c)$ are immediate consequences of Newmann's result \cite[14.5.11]{Rob}. The equivalences \ $(d) \Leftrightarrow (e)$ and $(d) \Leftrightarrow (f)$ were proved in \cite[Corollary 1.1]{BNR} and \cite[Theorem A]{BNR}, respectively. It is clear that $(b)$ implies $(f)$. Finally, if part $(f)$ holds then, from the decomposition (\ref{eq:decomposition}) and itens $(d)$, $(e)$, we obtain $(a)$. The proof is complete. 
\end{proof}

\noindent{\bf Acknowledgements.} The authors wish to thank I. Snopche for interesting discussions. This work was partially supported by FAPDF - Brazil, Grant: 0193.001344/2016.


\begin{thebibliography}{10}

\bibitem{BNR} R. Bastos, I.\,N. Nakaoka and N.\,R. Rocco, {\it Finiteness conditions for the non-abelian tensor product of groups}, Monatsh. Math. {\bf 187} (2018) pp. 603--615.

\bibitem{BR1} R. Bastos and N.\,R. Rocco, {\it The non-abelian tensor square of residually finite groups}, Monatsh. Math., {\bf 183} (2017) pp. 61--69.   

\bibitem{BR2} R. Bastos and N.\,R. Rocco, {\it Non-abelian tensor product of residually finite groups}, S\~ao Paulo J. Math. Sci., {\bf 11} (2017) pp. 361--369.   

\bibitem{BFM} R.\,D. Blyth, F. Fumagalli and M. Morigi, {\it Some structural results on the non-abelian tensor square of groups}, J. Group Theory, {\bf 13} (2010) pp. 83--94.

\bibitem{BL} R. Brown, and J.-L. Loday, {\it Van Kampen theorems for diagrams of spaces}, Topology, {\bf 26} (1987) pp. 311--335.

\bibitem{Ellis} G. Ellis, {\em The non-abelian tensor product of finite groups is finite}, J. Algebra, {\bf 111} (1987) pp. 203--205.

\bibitem{EL} G. Ellis and F. Leonard, {\em Computing Schur multipliers and tensor products of finite groups},  Proc. Royal Irish Acad., {\bf 95A} (1995) pp. 137--147.

\bibitem{K} L.-C. Kappe, {\em Nonabelian tensor products of groups:  the commutator connection}, Proc. Groups St. Andrews 1997 at Bath, London Math. Soc. Lecture Notes, {\bf 261} (1999) 447--454.

\bibitem{LT} M. Ladra and V.\,Z. Thomas, {\em Two generalizations of the nonabelian tensor product}, J. Algebra, {\bf 369} (2012) pp. 96--113.  

\bibitem{M} P. Moravec, {\em The exponents of nonabelian tensor products of groups}, J. Pure Appl. Algebra, {\bf 212} (2008)  pp. 1840--1848.

\bibitem{Nak} I.\,N. Nakaoka,  {\it Non-abelian tensor products of solvable groups}, J. Group Theory, {\bf 3} (2000) pp. 157--167.

\bibitem{NR01} I.\,N. Nakaoka and N.\,R. Rocco, Nilpotent actions on non-abelian tensor products of groups, Matem\'atica Contempor\^anea, {\bf 21} (2001) pp. 223--238. 

\bibitem{NR} I.\,N. Nakaoka and N.\,R. Rocco, {\em A survey of non-abelian tensor products of groups and related constructions}, Bol. Soc. Paran. Mat., {\bf 30} (2012) pp. 77--89.

\bibitem{NP} M. Parvizi and P. Niromand, {\it On the structure of groups whose exterior or tensor square is a $p$-group}, J. Algebra, {\bf 352} (2012) pp. 347--353.

\bibitem{Rob} D.\,J.\,S. Robinson,
\textit{A course in the theory of groups}, 2nd edition, Springer-Verlag, New York, 1996.

\bibitem{NR1} N.\,R. Rocco, {\it On a construction related to the non-abelian 
tensor square of a group}, Bol. Soc. Brasil Mat., {\bf 22} (1991) pp. 63--79.

\bibitem{NR2} N.\,R. Rocco, {\it A presentation for a crossed embedding of finite solvable groups}, Comm. Algebra {\bf 22} (1994) pp. 1975--1998.

\bibitem{T} V.\,Z. Thomas, {\em The non-abelian tensor product of finite groups is finite: a Homology-free proof}, Glasgow Math. J. {\bf 52}, (2010) pp. 473--477.

\bibitem{Vis} M.\,P. Visscher, On the nilpotency class and solvability lenght of nonabelian tensor products of groups, Arch. Math. {\bf 73} (1999)  pp. 161--171.  


\bibitem{W} J. Wiegold, Groups with boundedly finite classes of conjugate elements,   Proc. R. Soc. Lond. Ser. A Math. Phys. Eng. Sci., {\bf 238} (1957) pp. 389--401. 

\bibitem{ze1} E. Zel'manov,
{\it The solution of the restricted Burnside problem for groups of odd exponent}, Math. USSR Izv., {\bf 36} (1991) pp. 41--60.

\bibitem{ze2} E.\,I. Zel'manov,
{\it The solution of the restricted Burnside problem for 2-groups}, Math. Sb., {\bf 182} (1991) pp. 568--592.
\end{thebibliography}
\end{document}